\documentclass[11pt]{article}

\usepackage[margin=1in]{geometry}
\usepackage{graphicx} 
\usepackage{amssymb}
\usepackage{amsmath}
\usepackage{amsthm}
\usepackage{color}

\newtheorem{theorem}{Theorem}

\newtheorem{definition}{Definition}

\newtheorem{property}{Property}
\newtheorem{corollary}{Corollary}
\newtheorem{remark}{Remark}

\newcommand{\R}{\mathbb{R}}
\newcommand{\T}{\mathcal{T}}
\newcommand{\norm}[1]{\left\lVert#1\right\rVert}
\newcommand{\itp}{\hat{\otimes}_\varepsilon}

\title{Manifold Function Encoder: Identifying Different Functions Defined on Different Manifolds}

\author{Jun Hu\footnote{School of Mathematical Sciences, Peking University, Beijing 100871, China, and Chongqing Research Institute of Big Data, Peking University, Chongqing 401329, China (hujun@math.pku.edu.cn).} , Pengzhan Jin\footnote{National Engineering Laboratory for Big Data Analysis and Applications, Peking University, Beijing 100871, China, and Chongqing Research Institute of Big Data, Peking University, Chongqing 401329, China (jpz@pku.edu.cn).} , Weijun Zhang\footnote{School of Mathematical Sciences, Peking University, Beijing 100871, China (zhangweijun@stu.pku.edu.cn).}}

\date{}

\begin{document}

\maketitle

\begin{abstract}
We propose the Manifold Function Encoder (MFE) for identifying different functions defined on different manifolds. Both a manifold in Euclidean space and a function defined on this manifold can be viewed as bounded linear functionals on a suitable space of continuous functions. From this perspective, we treat manifold functions as elements of the dual space. By expanding them in the dual space based on appropriate approximating sequence of bases, we obtain a corresponding method for encoding manifold functions, that is MFE. Especially, we prove that MFE achieves super-algebraic convergence based on smooth bases commonly used in spectral methods, such as Legendre polynomials and Fourier basis. We further extend MFE to handle more complex cases, including joint manifold functions of different dimensions and manifold functions with different measures. In addition, we show the approximation theory for MFE-based operator learning, in particular learning the solution mappings of PDEs defined on varying domains, together with several numerical experiments including the 2-d Poisson equation and the 3-d elasticity problem on the real-world bearing.
\end{abstract}

\section{Introduction}
Most scientific and engineering problems boil down to the solution of partial differential equations (PDEs). To date, a variety of effective traditional numerical methods have been developed for PDEs, including the finite element method \cite{brenner2008mathematical}, the finite difference method \cite{smith1985numerical}, the spectral method \cite{BERNARDI1997209,shen2011spectral}, and others. These traditional computational methods are highly effective and widely used; however, they still possess certain limitations. For instance, they are often mesh-dependent or require regular solution domains, and struggle with high-dimensional problems and large-scale computational problems. To overcome the limitations mentioned above, there have been increasing efforts to solve PDEs using neural network-based approaches, driven by advances in artificial intelligence. Motivated by the strong expressive power of neural networks \cite{arora2016understanding,cybenko1989approximation,e2019barron,hornik1989multilayer,jin2024shallow,siegel2022high}, researchers have attempted to approximate PDE solutions using neural networks, subject to either strong-form or weak-form constraints, i.e. the PINNs \cite{karniadakis2021physics,lu2021physics,raissi2019physics} and the deep Ritz method \cite{yu2018deep}. While this class of methods has achieved notable success in inverse problems, they remain at a significant disadvantage compared to traditional numerical methods for forward problems, particularly in terms of solving speed. To achieve rapid solving, a new data-driven paradigm for learning solution operators of PDEs has been proposed, with typical examples including DeepONet \cite{lu2019deeponet,lu2021learning} and FNO \cite{li2020fourier}. By learning from data pairs of ``input conditions – output solutions'' of PDEs, they can quickly provide predicted solutions for any new given input conditions via a forward pass through the neural network. The field of operator learning has witnessed a remarkable boom in recent years, with a proliferation of novel neural network architectures, such as \cite{gupta2021multiwavelet,he2023mgno,jin2022mionet,rahman2022u,raonic2023convolutional,wu2023solving}. Among these architectures, the ONets series \cite{jin2022mionet,lu2021learning} stands out with an interesting property: the solutions they predict are represented by neural networks (trunk net), thereby allowing for convenient automatic differentiation of the predicted solutions. This characteristic enables the direct construction of loss functions from the PDE itself, thereby reducing reliance on data. Methods leveraging this advantage are referred to as physics-informed DeepONet/MIONet \cite{wang2021learning,zheng2023state}. Beyond direct prediction, operator learning has also been employed to accelerate traditional numerical computation methods. For classical second-order elliptic problems, an accelerated convergence theory for hybrid methods combining operator learning with traditional smoothing iterations has been established \cite{hu2025hybrid,zhang2022hybrid}. Whether for direct prediction or for accelerating traditional methods, there is a compelling need for a more universal operator learning model capable of handling more general PDE input conditions. Given that classical operator learning methods typically require a fixed PDE domain, we aim to extend them to arbitrarily varying and complex domains.

A highly representative strategy is deformation. The objective of these methods is to map a collection of topologically equivalent domains onto a canonical domain for parameterization (or referred to as ``encoding''). The deformation method based on FNO has yielded Geo-FNO \cite{li2022fourier}. Since FNO requires the input domains to be rectangles, it encounters obstacles when processing classes of domains that are not diffeomorphic to a rectangle. Meanwhile, the MIONet-based deformation approaches lift this restriction, including the D2D\&D2E \cite{xiao2024deformation,xiao2024learning} and the DIMON \cite{yin2024dimon}. This is because they do not presume a specific shape for the input domains, thereby allowing any domain to be designated as the standard reference domain. Especially, \cite{xiao2024deformation} establishes a comprehensive deformation-based theoretical framework for learning solution mappings of PDEs defined on varying domains, including the convergence analysis. The deformation-based encoding strategy proves highly efficient for collections of domains that are topologically homeomorphic with limited variation. Moreover, the explicit parameterization it provides for the domains enables direct shape differentiation, which is particularly advantageous for shape optimization problems \cite{zhou2024ai}. However, this approach becomes ineffective when confronted with more general scenarios, such as cases where the domains possess distinct topological structures.

An alternative and more crucial strategy is extension. The idea behind this class of methods is quite simple: extend a function defined on a complex geometry to a larger rectangle, and then treat it as a fixed-domain operator learning problem on that rectangle. The simplest extension method is the characteristic function, that is, filling the area outside the domain with zero (perhaps incorporating some smoothing technique at the boundary for a gradual transition) \cite{li2023geometry,liu2023domain}. Another popular extension method is the signed distance function (SDF) \cite{duvall2025discretization,he2024geom,park2019deepsdf,ye2025pdeformer}. In addition to explicit extension encoding, there are also implicit extension encoding methods, such as implicit neural representation (INR) \cite{serrano2023operator,sitzmann2020implicit}, and other techniques utilizing mesh or point clouds \cite{deng2024geometryguided,gao2025generative,liu2024laflownet,pfaff2020learning,qi2017pointnet,qi2017pointnet++,zeng2025point}. The extension-based encoding strategy imposes no requirements on the topological structure of the geometry (manifold), making it a universal encoding method. However, the price of universality is low encoding efficiency. In particular, for explicit extensions such as the SDF, the required dimensionality of the encoded vector becomes prohibitive in high-dimensional spaces. Thus, the pressing challenge is to devise an encoding method with high efficiency.

In this work, we propose an encoding scheme based on the extension strategy, namely Manifold Function Encoder (MFE), and theoretically prove its super-algebraic convergence. It stands as the first efficient encoder for manifold functions with theoretical guarantees. In the following section, we will also demonstrate a problem in three-dimensional space where promising results are achieved with a remarkably low encoding dimension via MFE.

\textbf{Contributions.} We list the contributions of this work as follows: (i) To encode different functions defined on different manifolds, we propose the Manifold Function Encoder (MFE) and prove that MFE achieves super-algebraic convergence based on smooth bases commonly used in spectral
methods, such as Legendre polynomials and Fourier basis. (ii) We show the approximation theory for MFE-based operator learning, in particular learning the solution mappings of PDEs defined on varying domains. (iii) We further extend MFE to handle more complex cases, including encoding the union of manifold functions of different dimensions (joint manifold functions), encoding manifold functions with different measures (related to unstructured mesh and non-uniform point cloud in practice), and dealing with multiple inputs and outputs operator learning problems. (iv) We present several numerical experiments to verify our results, especially a 3-d elasticity problem on the real-world bearing.

The remainder of this paper is organized as follows. We begin with defining the mathematical concept of manifold functions we study in Section \ref{sec:problem_setup}. In Section \ref{sec:MFE}, we formally propose the Manifold Function Encoder (MFE) and prove the super-algebraic convergence property of MFE. Subsequently, we show the approximation theory for MFE-based operator learning in Section \ref{sec:MFE_operator_learning}. Section \ref{sec:further_developments} presents further developments for MFE. Several numerical examples verifying the theoretical findings are shown in Section \ref{sec:numerical_experiments}. Finally, we conclude in Section \ref{sec:conclusions} with a summary.

\section{Problem Setup}\label{sec:problem_setup}
To make it easier to follow, let us begin by looking at the simplified situation. We will then move on to the complete formulation in the following section. Consider the union of a finite number of $k$-dimensional Lipschitz manifolds contained in $V:=[0,1]^d$, $0\leq k\leq d$, precisely define
\begin{equation}\label{eq:def_Mk}
\begin{split}
    \mathcal{M}_k:=\{\cup_{i=1}^{m}M_i:&M_i\subset V{\rm\ is\ a\ compact\ }k{\rm -dimensional\ Lipschitz\ manifold\ }(\mathcal{H}^k(M_i)<\infty),\\ &M_i\cap M_j{\rm\ is\ manifold\ of\ dimension\ less\ than\ }k{\rm\ or\ empty\ set}, i\neq j,m\geq 1\},
\end{split}
\end{equation}
where $\mathcal{H}^k$ is the $k$-dimensional Hausdorff measure, and
\begin{equation}
    X_k:=\{(M,f_M):M\in\mathcal{M}_k,f_M\in L^2(M)\}.
\end{equation}
Note that when $k=0$, we regard $M\in\mathcal{M}_0$ as a finite set of some points in $V$ (whose $\mathcal{H}^0$ measure equals to the number of the points), and $f_M\in L^2(M)$ as a real value function on these points.  

To approximately represent a manifold function $(M,f_M)\in X_k$ more efficiently with a finite-dimensional vector, we need to find a continuous mapping $\Phi_n:X_k\to \R^{\kappa(n)}$ ($\kappa(n)$ is a monotonically increasing function that maps to the positive integers) such that $\Phi_n(M,f_M)\in\R^{\kappa(n)}$ captures as much information about $(M,f_M)$ as possible. At the very least, we have to find an encoder with certain convergence properties, such that when the dimension of the encoded vector is sufficiently large, the amount of information it captures approaches the manifold function itself asymptotically. The challenge is twofold: first, how to construct an effective encoder; second, how to establish a suitable mathematical framework for its analysis.

\section{Manifold Function Encoder}\label{sec:MFE}
In this section, we propose the Manifold Function Encoder (MFE) for identifying different functions defined on different manifolds, and further provide its convergence theory. We begin with the definition of the used approximating sequence of bases, which is necessary for constructing MFE.
\begin{definition}[approximating sequence of bases]
    Let $\kappa$ be a monotonically increasing function that maps positive integers to positive integers, $\{\phi_{n,m}\}_{1\leq n,1\leq m\leq \kappa(n)}\subset C(V)$ be a series of linearly independent basis functions, satisfying
    \begin{equation}
        \lim_{n\to\infty}\inf_{v\in \operatorname{span}\{\phi_{n,m}\}_{m=1}^{\kappa(n)}}\norm{\phi-v}_{C(V)}=0,\quad \forall \phi\in C(V).
    \end{equation}
\end{definition}
\noindent Classical orthogonal polynomials, Fourier basis, and finite elements are all examples of approximating sequences of bases.

With an approximating sequence of bases $\{\phi_{n,m}\}_{1\leq n,1\leq m\leq \kappa(n)}\subset C(V)$, we can construct the encoder. Given any
\begin{equation}
    (M,f_M)\in X_k,
\end{equation}
let
\begin{equation}
    \Phi_n^1:\mathcal{M}_k\to\R^{\kappa(n)},\quad \Phi_n^1(M):=\left(\int_{M}1\cdot \phi_{n,m}d\mathcal{H}^k\right)_{1\leq m\leq\kappa(n)}\in\R^{\kappa(n)},
\end{equation}
and
\begin{equation}
    \Phi_n^2:X_k\to\R^{\kappa(n)},\quad\Phi_n^2(M,f_M):=\left(\int_{M}f_M\cdot\phi_{n,m}d\mathcal{H}^k\right)_{1\leq m\leq\kappa(n)}\in\R^{\kappa(n)}.
\end{equation}
The final encoder is defined as
\begin{equation}
    \Phi_n:X_k\to\R^{2\kappa(n)},\quad \Phi_n(M,f_M):=(\Phi_n^1(M),\Phi_n^2(M,f_M))\in\R^{2\kappa(n)}.
\end{equation}
Taking the case of Legendre polynomials as an example, the encoder can be written as
\begin{equation}
    \begin{cases}
        \Phi_n^1:\mathcal{M}_k\to\R^{n^d},\quad \Phi_n^1(M):=\left(\int_{M}1\cdot (\ell_{i_1}\otimes\cdots\otimes \ell_{i_d})d\mathcal{H}^k\right)_{0\leq i_1,...,i_d\leq n-1}\in\R^{n^d}, \\
        \Phi_n^2:X_k\to\R^{n^d},\quad\Phi_n^2(M,f_M):=\left(\int_{M}f_M\cdot (\ell_{i_1}\otimes\cdots\otimes \ell_{i_d})d\mathcal{H}^k\right)_{0\leq i_1,...,i_d\leq n-1}\in\R^{n^d},\\
        \Phi_n:X_k\to\R^{2n^d},\quad \Phi_n(M,f_M):=(\Phi_n^1(M),\Phi_n^2(M,f_M))\in\R^{2n^d},
    \end{cases}
\end{equation}
where $\{\ell_i\}_{i=0}^\infty$ is the Legendre polynomials on $[0,1]$, and $\otimes$ is the tensor product.

We next establish the convergence theory for such encoder. 

\subsection{Convergence Theory}
To begin with the convergence theory, we first show the definition of $(r,s,t)$-convergence sequence.
\begin{definition}[$(r,s,t)$-convergence sequence]
    We say $\{\phi_{n,m}\}_{1\leq n,1\leq m\leq \kappa(n)}$ is a $(r,s,t)$-convergence sequence if $\{\phi_{n,m}\}\subset C(V)\cap H^{s}(V)$ linearly independent and
    \begin{equation}
        \sup_{\norm{\phi}_{C^r(V)}\leq 1}\norm{\phi -\Pi_n^s\phi}_{C(V)}=\mathcal{O}\left(\frac{1}{n^t}\right),
    \end{equation}
    where 
    \begin{equation}
        \Pi_n^s\phi:=\operatorname*{arg\,min}_{v\in\operatorname{span}\{\phi_{n,m}\}_{m=1}^{\kappa(n)}}\norm{\phi -v}_{H^s(V)}
    \end{equation}
    is the $H^s$-orthogonal projection of $\phi$ on $\operatorname{span}\{\phi_{n,m}\}_{m=1}^{\kappa(n)}$.
\end{definition}
\noindent A sequence is $(r,s,t)$-convergence means its orthogonal projection series in $H^{s}(V)$ uniformly converges on the unit closed ball of $C^r(V)$ with $C(V)$-error of order $t$. We present a useful example.
\begin{property}
    The approximating sequence of bases
    \begin{equation}
        \phi_{n,(i_1,...,i_d)}:=\ell_{i_1}\otimes\cdots\otimes\ell_{i_d},\quad 0\leq i_1,...,i_d\leq n-1, 
    \end{equation}
    is $(r,s,r-s)$-convergence for any $\frac{d}{2}< s < r$, where $\{\ell_i\}_{i=0}^\infty$ is the Legendre polynomials on $[0,1]$, and $\kappa(n)=n^d$ in this case.
\end{property}
\begin{proof}
    Consider the test function $\phi$ in $C^{r}(V)$. According to Sobolev embedding theorem and $s>\frac{d}{2}$, the sobolev space $H^{s}(V)$ can be continuously embedded into $C(V)$ . More precisely, we have
    \begin{equation}
        \norm{\phi-\Pi_n^{s}\phi}_{C(V)} \lesssim \norm{\phi-\Pi_n^{s}\phi}_{H^{s}(V)}.
    \end{equation}
    As $r>s>0$ and function $\phi \in C^r(V)\subset H^r(V)$, by Theorem~7.3 of \cite{BERNARDI1997209}, we have the following approximation estimate
    \begin{equation}
        \norm{\phi-\Pi_n^{s}\phi}_{H^{s}(V)} \lesssim n^{s-r}\norm{\phi}_{H^r(V)}.
    \end{equation}
    Combined again with $C^r(V) \subset H^r(V)$, we obtain 
    \begin{equation}
        \norm{\phi-\Pi_n^{s}\phi}_{C(V)} \lesssim n^{s-r}\norm{\phi}_{C^r(V)},
    \end{equation}
    for any $\frac{d}{2}<s<r$ and function $\phi \in C^r(V).$
\end{proof}

Return to the topic on manifold functions, next we characterize the space $X_k$ by the dual space of $C(V)$. Precisely, define the linear functionals
\begin{equation}
    \begin{cases}
        \langle M,\phi\rangle:=\int_M1\cdot \phi d\mathcal{H}^k,\quad \forall \phi\in C(V),\\
        \langle f_M,\phi\rangle:=\int_Mf_M\cdot \phi d\mathcal{H}^k,\quad\forall \phi\in C(V),
    \end{cases}
\end{equation}
for any $(M,f_M)\in X_k$. It is clear that
\begin{equation}
\begin{cases}
    |\langle M,\phi\rangle|=|\int_M1\cdot \phi d\mathcal{H}^k|\leq \mathcal{H}^k(M)\cdot\norm{\phi}_{C(V)},\\
    |\langle f_M,\phi\rangle|=|\int_Mf_M\cdot \phi d\mathcal{H}^k|\leq\norm{f_M}_{L^2(M)}\cdot\norm{\phi}_{L^2(M)}\leq\sqrt{\mathcal{H}^k(M)}\cdot\norm{f_M}_{L^2(M)}\cdot\norm{\phi}_{C(V)},
\end{cases}
\end{equation}
thus $M$ and $f_M$ can be regarded as bounded linear functionals on $C(V)$, i.e.,
\begin{equation}
    M\in (C(V))',\quad f_M\in (C(V))'.
\end{equation}
Furthermore, if $M_1=M_2$ and $f_{M_1}^1=f_{M_2}^2$ in $(C(V))'$, it is straightforward to verify that $M_1=M_2$ in $X_k$ and $f_{M_1}^1=f_{M_2}^2$ in $L^2(M_1)$. Now we have a well-defined embedding
\begin{equation}
\begin{split}
    X_k&\hookrightarrow (C(V))'\times (C(V))',\\
    (M,f_M)&\mapsto(\langle M,\cdot\rangle,\langle f_M,\cdot\rangle).
\end{split}
\end{equation}
With such embedding, we obtain a subspace endowed with the induced metric, i.e.,
\begin{equation}
    (X_k, d_{X_k}),\quad d_{X_k}((M_1,f_{M_1}^1),(M_2,f_{M_2}^2)):=\norm{M_1-M_2}_{(C(V))'}+\norm{f_{M_1}^1-f_{M_2}^2}_{(C(V))'}.
\end{equation}
Note that there are canonical compact embeddings
\begin{equation}
    (C(V))'\times (C(V))'\hookrightarrow (C^1(V))'\times (C^1(V))'\hookrightarrow (C^2(V))'\times (C^2(V))'\hookrightarrow\cdots
\end{equation}
induced by restricted functionals and operator norms.

Based on the concepts, we present the approximation property for the proposed encoder. Slightly differing from the previous definition of the encoder on $X_k$, here we consider the extended encoder as
\begin{equation}
\begin{split}
    \Phi_n:(C(V))'\times (C(V))'&\to\R^{2\kappa(n)} \\
    (f_1,f_2)&\mapsto\left((\langle f_1,\phi_{n,m}\rangle)_{m=1}^{\kappa(n)},(\langle f_2,\phi_{n,m}\rangle)_{m=1}^{\kappa(n)}\right)
\end{split}
\end{equation}
together with a decoder
\begin{equation}
\begin{split}
    \Psi_n^{r,s}:\R^{2\kappa(n)}&\to (C^r(V))'\times (C^r(V))' \\
    \left((\alpha_{m})_{m=1}^{\kappa(n)},(\beta_{m})_{m=1}^{\kappa(n)}\right)&\mapsto \left(\sum_{m=1}^{\kappa(n)}\alpha_{m}e_{n,m}^*, \sum_{m=1}^{\kappa(n)}\beta_{m}e_{n,m}^*\right)
\end{split}
\end{equation}
where
\begin{equation}
\langle e_{n,m}^*,\cdot\rangle:=\langle \phi_{{n,m}}^*,\Pi_n^s(\cdot)\rangle,\quad e_{n,m}^*\in(C^r(V))',
\end{equation}
for $\{\phi_{n,m}\}\subset C(V)\cap H^{s}(V)$, $0\leq s\leq r$. Here $\phi_{n,m}^*$ is the dual basis of $\phi_{n,m}$ with respect to the finite dimensional space $\operatorname{span}\{\phi_{n,m}\}_{m=1}^{\kappa(n)}$.
\begin{theorem}[convergence rate of MFE]
    Suppose that $(X_k, d_{X_k})$ is the metric space of manifold functions defined above, $K$ is a compact set in $X_k$. If $\{\phi_{n,m}\}_{1\leq n,1\leq m\leq \kappa(n)}$ is a $(r,s,t)$-convergence sequence inducing encoder $\{\Phi_{n}\}$ and decoder $\{\Psi_{n}^{r,s}\}$, then
\begin{equation}
 \norm{P_n^{r,s}-I}_{C(K,(C^r(V))'\times (C^r(V))')} = \mathcal{O}\left(\frac{1}{n^{t}}\right),
\end{equation}
where $P_n^{r,s}:=\Psi_n^{r,s}\circ\Phi_n$ is the reconstruction mapping and $I$ is the identity mapping (embedding).
\end{theorem}
\begin{proof}
Since $\{\phi_{n,m}\}$ is $(r,s,t)$-convergence, we have
\begin{equation}\label{eq:rs_convergence}
    \sup_{\norm{\phi}_{C^r(V)}\leq 1}\norm{\phi -\Pi_n^{s}\phi}_{C(V)}=\mathcal{O}\left(\frac{1}{n^t}\right).
\end{equation}
Recall that the reconstruction mapping is
\begin{equation}
    P_n^{r,s}(M,f_M)=\left(P_n^1(M), P_n^2(M,f_M)\right):=\left(\sum_{m=1}^{\kappa(n)}\langle M,\phi_{n,m}\rangle e_{n,m}^*,\sum_{m=1}^{\kappa(n)}\langle f_M,\phi_{n,m}\rangle e_{n,m}^*\right),
\end{equation}
for $(M,f_M)\in X_k$. Let $\phi \in C^r(V)$ be a test function, then
\begin{equation}
\begin{split}
        \langle P_n^1(M), \phi \rangle &= \left\langle \sum_{m=1}^{\kappa(n)}\langle M,\phi_{n,m}\rangle e_{n,m}^*, \phi \right\rangle \\
        &= \sum_{m=1}^{\kappa(n)}\langle M,\phi_{n,m}\rangle\cdot\langle e_{n,m}^*,\phi\rangle\\&=\langle M, \sum_{m=1}^{\kappa(n)}\langle e_{n,m}^*,\phi\rangle\phi_{n,m} \rangle\\&=\langle M, \Pi_n^{s}\phi \rangle.
\end{split}
\end{equation}
Therefore
    \begin{equation}
        \langle M - P_n^1(M), \phi \rangle = \langle M, \phi \rangle - \langle M, \Pi_n^{s}\phi \rangle = \langle M, \phi - \Pi_n^{s}\phi \rangle,
    \end{equation}
    and similarly
    \begin{equation}
        \langle f_M - P_n^2(M,f_M), \phi \rangle = \langle f_M, \phi - \Pi_n^{s}\phi \rangle.
    \end{equation}
Since $M$ and $f_M$ are bounded linear functionals on $C(V)$, we have
\begin{equation}
    \begin{cases}
        |\langle M, \phi - \Pi_n^{s}\phi \rangle| \leq \norm{M}_{(C(V))'} \norm{\phi - \Pi_n^{s}\phi}_{C(V)}, \\
        |\langle f_M, \phi - \Pi_n^{s}\phi \rangle| \leq \norm{f_M}_{(C(V))'} \norm{\phi - \Pi_n^{s}\phi}_{C(V)},
    \end{cases}
\end{equation}    
subsequently
\begin{equation}
\begin{split}
    \norm{M-P_n^1(M)}_{(C^r(V))'}&=\sup_{\norm{\phi}_{C^r(V)}\leq 1}|\langle M - P_n^1(M),\phi \rangle|\\
    &=\sup_{\norm{\phi}_{C^r(V)}\leq 1}|\langle M,\phi-\Pi_n^{s}\phi \rangle|\\
    &\leq \norm{M}_{(C(V))'}\sup_{\norm{\phi}_{C^r(V)}\leq 1}\norm{\phi - \Pi_n^{s}\phi}_{C(V)},
\end{split}
\end{equation}
and
\begin{equation}
    \norm{f_M - P_n^2(M,f_M)}_{(C^r(V))'} \leq \norm{f_M}_{(C(V))'} \sup_{\norm{\phi}_{C^r(V)} \leq 1} \norm{\phi - \Pi_n^{s}\phi}_{C(V)}.
\end{equation}
    
Combining these estimates, we derive that
\begin{equation}
\begin{split}
    &\norm{(M,f_M) - P_n^{r,s}(M,f_M)}_{(C^r(V))'\times (C^r(V))'}\\
    =&\norm{M-P_n^1(M)}_{(C^r(V))'}+\norm{f_M - P_n^2(M,f_M)}_{(C^r(V))'}\\
    \leq&(\norm{M}_{(C(V))'} + \norm{f_M}_{(C(V))'}) \sup_{\norm{\phi}_{C^r(V)} \leq 1} \norm{\phi - \Pi_n^{s}\phi}_{C(V)}\\
    =&\norm{(M,f_M)}_{(C(V))'\times (C(V))'}\sup_{\norm{\phi}_{C^r(V)} \leq 1} \norm{\phi - \Pi_n^{s}\phi}_{C(V)}.
\end{split}
\end{equation}
As $K\subset X_k$ is compact, $\norm{(M,f_M)}_{(C(V))'\times (C(V))'}$ is bounded on $K$, i.e.
\begin{equation}
    \sup_{(M,f_M)\in K}\norm{(M,f_M)}_{(C(V))'\times (C(V))'}<\infty,
\end{equation}
thus
\begin{equation}
\begin{split}
     &\norm{P_n^{r,s}-I}_{C(K,(C^r(V))'\times (C^r(V))')}\\
    =&\sup_{(M,f_M)\in K}\norm{(M,f_M) - P_n^{r,s}(M,f_M)}_{(C^r(V))'\times (C^r(V))'}\\ \leq&\sup_{(M,f_M)\in K}\norm{(M,f_M)}_{(C(V))'\times (C(V))'}\sup_{\norm{\phi}_{C^r(V)} \leq 1} \norm{\phi - \Pi_n^{s}\phi}_{C(V)}.
\end{split}
\end{equation}
Lastly, Eq. \eqref{eq:rs_convergence} leads to $\norm{P_n^{r,s}-I}_{C(K,(C^r(V))'\times (C^r(V))')} = \mathcal{O}\left(\frac{1}{n^{t}}\right)$.
\end{proof}

As aforementioned, the approximating sequence of bases of Legendre polynomials has been proved with $(r,s,r-s)$-convergence property, so that a direct corollary can be derived.
\begin{corollary}[MFE based on Legendre polynomials]
Suppose that $(X_k, d_{X_k})$ is the metric space of manifold functions defined above, $K$ is a compact set in $X_k$. Let
\begin{equation}
    \{\phi_{n,(i_1,...,i_d)}:=\ell_{i_1}\otimes\cdots\otimes\ell_{i_d}\}_{1\leq n,0\leq i_1,...,i_d\leq n-1}
\end{equation}
be the approximating sequence of bases of Legendre polynomials, which induces encoder $\{\Phi_{n}\}$ and decoder $\{\Psi_{n}^{r,s}\}$. Then for any $\frac{d}{2}<s<r$, we have
\begin{equation}
\norm{P_n^{r,s}-I}_{C(K,(C^{r}(V))'\times (C^{r}(V))')} = \mathcal{O}\left(\frac{1}{n^{r-s}}\right),
\end{equation}
where $P_n^{r,s}:=\Psi_n^{r,s}\circ\Phi_n$ is the reconstruction mapping and $I$ is the identity mapping.
\end{corollary}
This corollary points out that MFE based on Legendre polynomials exhibits \textbf{super-algebraic} convergence. 
\begin{remark}
    For the sequence of Legendre polynomials, taking into account the dimension of encoded vector, i.e. $N:=2n^d$, the convergence rate with respect to $N$ is $\mathcal{O}(N^{-\frac{r-s}{d}})$.
\end{remark}
\begin{remark}
    According to Theorem~2.2 of \cite{canuto1982approximation}, we obtain an alternative result on the approximation error of Legendre polynomials under the $L^2$ projection
     \begin{equation}
        \norm{\phi-\Pi_n^{0}\phi}_{H^{s}(V)} \lesssim n^{2s-r-\frac{1}{2}}\norm{\phi}_{H^r(V)},
    \end{equation}
    where $r>s>1$. Hence, the sequence of Legendre polynomials is also $(r,0,2s-r-\frac{1}{2})$-convergence for $r>s>\max\{1, \frac{d}{2}\}$, leading to the MFE with convergence rate of $\mathcal{O}(n^{2s-r-\frac{1}{2}})$. Accordingly, we take the decoder $\Psi_n^{r,0}$ constructed via the $L^2$ projection, instead of $\Psi_n^{r,s}$ constructed via the $H^s$ projection.
\end{remark}
\begin{remark}\label{rem:feasible}
    Note that in the preceding analysis, the only property of $X_k$ we used is its injective embedding into $(C(V))'\times(C(V))'$. Hence, any encoding scheme that admits such an injective embedding is feasible—for example, we can replace the constant ``$1$'' in $\Phi_n^1$ with any non-zero function that does not change sign in $V$.
\end{remark}

\section{MFE for Operator Learning}\label{sec:MFE_operator_learning}
In this section, we discuss the application of MFE on operator learning.
\begin{theorem}\label{thm:mfe_for_banach}
    Assume that $(X_k, d_{X_k})$ is the metric space of manifold functions defined above, $K$ is a compact set in $X_k$, $Y$ is a Banach space, $\{\phi_{n,m}\}_{1\leq n,1\leq m\leq \kappa(n)}$ is a $(r,s,t)$-convergence ($t>0$) sequence inducing the MFE $\{\Phi_n,\Psi_n^{r,s}\}$. Suppose that 
    \begin{equation}
        \mathcal{G}:K\to Y
    \end{equation}
    is a continuous mapping, then for any $\epsilon>0$, there exist positive integers $p,q$, a continuous mapping $\tilde{\mathcal{G}}:\R^{2\kappa(q)}\to\R^{p}$ and $u\in C(\R^p,Y)$ such that
    \begin{equation}\label{eq:app_general}
    \sup_{(M,f_M)\in K} \norm{\mathcal{G}(M,f_M) - u\circ\tilde{\mathcal{G}}(\Phi_{q}(M,f_M))}_Y< \epsilon.
    \end{equation}
    For convenience, we simplify the notation as
    \begin{equation}
    \begin{split}
        \hat{\mathcal{G}}:K&\to Y \\
        (M,f_M)&\mapsto u\circ\tilde{\mathcal{G}}(\Phi_{q}(M,f_M)),
    \end{split}
    \end{equation}
    thus Eq. \eqref{eq:app_general} can be rewritten as
    \begin{equation}
        \norm{\mathcal{G}-\hat{\mathcal{G}}}_{C(K,Y)}<\epsilon.
    \end{equation}
\end{theorem}
\begin{proof}
We prove the theorem in four steps. In this proof we denote $P_n^{r,s},\Psi_n^{r,s}$ by $P_n,\Psi_n$ for short.

1. Note that $K$ is compact and $\mathcal{G}\in C(K,Y)$, based on the injective tensor product \cite{ryan2002introduction}
\begin{equation}
    C(K,Y) = C(K) \itp Y,
\end{equation}
for any $\epsilon > 0$, there exist a positive integer $p$, continuous functions $\hat{g}_i \in C(K)$, and elements $u_i \in Y$ for $i = 1, \dots, p$, such that
\begin{equation}
    \sup_{v \in K} \norm{\mathcal{G}(v) - \sum_{i=1}^p \hat{g}_i(v) \cdot u_i}_Y < \frac{\epsilon}{2}.
\end{equation}
    
2. We assert that $g \circ P_n \to g$ uniformly on $K$ for any $g \in C((C^r(V))' \times (C^r(V))', \mathbb{R})$. From the previous theorem, we know that $P_n \to I$ uniformly on $K$, i.e.
    \begin{equation}
        \lim_{n \to \infty} \sup_{v \in K} \norm{P_n(v) - v}_{(C^r(V))' \times (C^r(V))'} = 0,
    \end{equation}
so that for any $\delta > 0$, there exists $N$ such that $P_n(K)$ is contained in a $\delta$-neighborhood of $K$ for $n \geq N$. Since $K$ is compact, it has a finite $\delta$-net for $K$, which then serves as a $2\delta$-net for $\bigcup_{n \geq N} P_n(K) \cup K$. As $P_1(K),...,P_{N-1}(K)$ are also compact, we can find a $2\delta$-net for $S := \bigcup_{n \geq 1} P_n(K) \cup K$, which means $S$ is a totally bounded set, consequently its closure $\overline{S}$ is compact.
    
Since $g$ is continuous on $(C^r(V))' \times (C^r(V))'$ and $\overline{S}$ is compact, $g$ is uniformly continuous on $\overline{S}$. Therefore, for any $\eta > 0$, there exists $\delta > 0$ such that $x, y \in \overline{S}$ and $\norm{x - y} < \delta$ leads to $|g(x) - g(y)| < \eta$. Given that $P_n \to I$ uniformly on $K$, for sufficiently large $n$, we have
\begin{equation}
    \sup_{v \in K} \norm{P_n(v) - v} < \delta,
\end{equation}
which implies
\begin{equation}
    \sup_{v \in K} |g(P_n(v)) - g(v)| < \eta.
\end{equation}
That shows $g \circ P_n \to g$ uniformly on $K$.
    
3. Now consider the functions $\hat{g}_i \in C(K)$ from the first step. Since $K$ is a closed subset of the metric space $(C^r(V))' \times (C^r(V))'$, there exist continuous extensions $\tilde{g}_i: (C^r(V))' \times (C^r(V))' \to \mathbb{R}$ such that $\tilde{g}_i|_K = \hat{g}_i$, according to Dugundji's extension theorem \cite{dugundji1951extension}. From the second step, we know that $\tilde{g}_i \circ P_n \to \tilde{g}_i$ uniformly on $K$, i.e.
    \begin{equation}
        \lim_{n \to \infty} \sup_{v \in K} |\tilde{g}_i(P_n(v)) - \tilde{g}_i(v)| = 0.
    \end{equation}
Hence for sufficiently large $q$ so that $\sup_{v \in K} |\tilde{g}_i(v) - \tilde{g}_i(P_q(v))| < \frac{\epsilon}{2p \max_j \norm{u_j}_Y}$ for all $i$ , we have
\begin{equation}
    \begin{split}
        &\sup_{v\in K}\norm{\mathcal{G}(v) - \sum_{i=1}^p \tilde{g}_i(P_q(v)) \cdot u_i}_Y \\
        \leq& \sup_{v\in K}\norm{\mathcal{G}(v) - \sum_{i=1}^p \tilde{g}_i(v) \cdot u_i}_Y + \sup_{v\in K}\norm{\sum_{i=1}^p [\tilde{g}_i(v) - \tilde{g}_i(P_q(v))] \cdot u_i}_Y \\
        \leq& \sup_{v\in K}\norm{\mathcal{G}(v) - \sum_{i=1}^p \tilde{g}_i(v) \cdot u_i}_Y + \sum_{i=1}^p \sup_{v\in K}|\tilde{g}_i(v) - \tilde{g}_i(P_q(v))| \cdot \norm{u_i}_Y\\ <&\frac{\epsilon}{2} + \frac{\epsilon}{2}=\epsilon.
    \end{split}
\end{equation}
    
4. Now define the mappings
    \begin{equation}
        \tilde{\mathcal{G}}: \mathbb{R}^{2\kappa(q)} \to \mathbb{R}^p, \quad \tilde{\mathcal{G}}(x) := (\tilde{g}_1(\Psi_q(x)), \dots, \tilde{g}_p(\Psi_q(x))),
    \end{equation}
    and
    \begin{equation}
        u: \mathbb{R}^p \to Y, \quad u(\alpha_1, \dots, \alpha_p):= \sum_{i=1}^p \alpha_i \cdot u_i,
    \end{equation}
which lead to
\begin{equation}
    \sup_{(M,f_M) \in K} \norm{\mathcal{G}(M,f_M) - u \circ \tilde{\mathcal{G}} \circ \Phi_q(M,f_M)}_Y < \epsilon,
\end{equation}
and complete the proof.
\end{proof}
Note that $u$ in above theorem can be flexibly selected. For example, in DeepONet/MIONet \cite{jin2022mionet,lu2021learning}, $u$ corresponds to the trunk net which is trainable.

\subsection{Learning Solution Mappings of PDEs}

As discussed above, we have presented an approximation theorem for learning the continuous mapping from the manifold function space to the Banach space, however, there is still a gap between such case and PDE problem. Let $X_{\tilde{k}}$ be another manifold function space maybe different from $X_k$, i.e. $\tilde{k}$ is not necessary equivalent to $k$. A solution mapping of parametric PDEs can be written as
\begin{equation}
    \mathcal{G}:X_k\to X_{\tilde{k}},
\end{equation}
and we need to explain why neural networks are able to learn such continuous mappings. Essentially, we aim to find an intermediate Banach space that can bridge Theorem \ref{thm:mfe_for_banach} and the above solution mapping. However, finding such an intermediate space is nontrivial; indeed, it may not even exist within the conventional classes of function spaces, such as Sobolev spaces. Nevertheless, for some simple cases, we can still manage to obtain some meaningful results.

\textbf{Case of $\tilde{k}=d$.} In this case, the solutions are defined on full-dimensional manifolds, thus it is easy to find an intermediate space. For example, define the extension mapping and the restriction mapping as
\begin{equation}
    E:X_{d}\to L^2(V),\quad E(M,u_M)(x):=\begin{cases}
        u_M(x),\quad &x\in M,\\
        0,\quad &x\notin M,
    \end{cases}
\end{equation}
and
\begin{equation}
    R:\mathcal{M}_d\times L^2(V)\to X_d,\quad R(M,u):=(M,u|_M).
\end{equation}
Subsequently we have a mapping chain as
\begin{equation}
\begin{split}
    &X_k\to X_{d}\xrightarrow{E} \mathcal{M}_{d}\times L^2(V)\xrightarrow{R} X_{d},\\ 
    &(M,f_M)\mapsto(J(M),u_{J(M)})\mapsto\\&(J(M), E(J(M),u_{J(M)}))\mapsto (J(M),u_{J(M)}=E(J(M),u_{J(M)})|_{J(M)})
\end{split}
\end{equation}
where $J:\mathcal{M}_k\to\mathcal{M}_d$ maps the input manifold to the solution manifold, which is a known mapping. Denote
\begin{equation}
    I_1(M,f_M):=M,\quad I_2(M,f_M):=f_M,
\end{equation}
then based on the mapping chain, we have an equivalent representation for $\mathcal{G}$, that is
\begin{equation}
    \mathcal{G}=R\circ(J\circ I_1,E\circ\mathcal{G}).
\end{equation}
As all the $E,R,J,I_1$ are known mappings, it is sufficient to learn the mapping
\begin{equation}
    E\circ\mathcal{G}:X_k\to L^2(V),
\end{equation}
which is indeed the case we have already solved in Theorem \ref{thm:mfe_for_banach}. The loss function for such learning problem can be written as
\begin{equation}
    L(\T):=\frac{1}{N}\sum_{i=1}^N\norm{I_2\circ R\circ(J\circ I_1,\hat{\mathcal{G}})(M_i,f_{M_i}^i)-u_{J(M_i)}^i}_{L^2(J(M_i))}^2,
\end{equation}
where $\T=\{(M_i,f_{M_i}^i),(J(M_i),u_{J(M_i)}^i)\}_{i=1}^N$ is the dataset, satisfying
\begin{equation}
    \mathcal{G}(M_i,f_{M_i}^i)=(J(M_i),u_{J(M_i)}^i),\quad 1\leq i\leq N,
\end{equation}
and $\hat{\mathcal{G}}:X_k\to L^2(V)$ is learned by neural networks, for approximating $E\circ\mathcal{G}$. A numerical example of this case is shown in Section \ref{sec:elasticity}.

\textbf{Case of continuous solutions.} Now we assume that $\tilde{k}\neq d$ but the solution has better regularity. Denote
\begin{equation}
    Y_{\tilde{k}}:=\{(M,u_M):M\in\mathcal{M}_{\tilde{k}},u_M\in C(M)\},
\end{equation}
where we consider continuous functions on manifolds. Assume that $E$ is an appropriate extension mapping and $R$ is the restriction mapping, satisfying
\begin{equation}
    E:Y_{\tilde{k}}\to C(V),\quad E(M,u_{M})|_{M}=u_M,
\end{equation}
and
\begin{equation}
    R:\mathcal{M}_{\tilde{k}}\times C(V)\to Y_{\tilde{k}},\quad R(M,u):=(M,u|_M).
\end{equation}
By Dugundji’s extension theorem \cite{dugundji1951extension}, such extension $E$ exists. Similarly, we can derive that
\begin{equation}
    \mathcal{G}=R\circ(J\circ I_1,E\circ\mathcal{G}),
\end{equation}
with
\begin{equation}
    E\circ\mathcal{G}:X_{k}\to C(V).
\end{equation}
Hence the learning problem of PDEs has been transformed to learning a mapping from $X_{k}$ to the Banach space $C(V)$, which is shown in Theorem \ref{thm:mfe_for_banach}.

\textbf{Remaining problems.} To complete the learning theory of parametric PDEs, there are several remaining problems. (i) We need to verify that both the extension mapping and the restriction mapping are continuous within the above defined topology. (ii) We need to verify that the solution mapping of a specific parametric PDEs problem, for example, the parametric Poisson equations, is indeed a continuous mapping, within the above defined topology. (iii) How to find a more suitable function space to handle the generic case, which is not limited to full-dimensional manifolds or continuous solutions. We leave them for future work.

\section{Further Developments}\label{sec:further_developments}
Here we show some further developments for MFE and operator learning.

\subsection{Joint Manifold Functions of Different Dimensions}\label{sec:joint}
Although our previous discussion focuses on manifolds with fixed dimension, all the results can be extended to the case of multiple dimensions. Define the set
\begin{equation}
    \mathcal{M}:=(\mathcal{M}_0\cup\{\varnothing\})\times(\mathcal{M}_1\cup\{\varnothing\})\times\cdots\times(\mathcal{M}_{d-1}\cup\{\varnothing\})\times(\mathcal{M}_d\cup\{\varnothing\}),
\end{equation}
and
\begin{equation}
\begin{split}
X:=\{(M_0,...,M_d,f_{M_0}^0,...,f_{M_d}^d):&(M_0,...,M_d)\in\mathcal{M}, f_{M_k}^k\in L^2(M_k){\rm \ for\ }M_k\neq\varnothing, \\ &f_{\varnothing}^k=\varnothing,0\leq k\leq d\}.
\end{split}
\end{equation}
For $(M,f_M)=(M_0,...,M_d,f_{M_0}^0,...,f_{M_d}^d)\in X$, the corresponding encoder is
\begin{equation}\label{eq:full_encoder}
    \begin{cases}
        \Phi_n^1:\mathcal{M}\to\R^{\kappa(n)},\quad \Phi_n^1(M):=\left(\sum_{0\leq k\leq d,M_k\neq\varnothing}\frac{1}{\mathcal{H}^k(M_k)}\int_{M_k}1\cdot \phi_{n,m}d\mathcal{H}^k\right)_{1\leq m\leq \kappa(n)}\in\R^{\kappa(n)}, \\
        \Phi_n^2:X\to\R^{\kappa(n)},\quad\Phi_n^2(M,f_M):=\left(\sum_{0\leq k\leq d,M_k\neq\varnothing}\frac{1}{\mathcal{H}^k(M_k)}\int_{M_k}f_{M_k}^k\cdot\phi_{n,m}d\mathcal{H}^k\right)_{1\leq m\leq \kappa(n)}\in\R^{\kappa(n)},\\
        \Phi_n:X\to\R^{2\kappa(n)},\quad \Phi_n(M,f_M):=(\Phi_n^1(M),\Phi_n^2(M,f_M))\in\R^{2\kappa(n)}.
    \end{cases}
\end{equation}
The analysis for such joint manifold functions is the same as before, by considering the functionals
\begin{equation}
    \begin{cases}
        \langle M,\phi\rangle:=\sum_{0\leq k\leq d,M_k\neq\varnothing}\frac{1}{\mathcal{H}^k(M_k)}\int_{M_k}1\cdot \phi d\mathcal{H}^k,\quad \forall \phi\in C(V),\\
        \langle f_M,\phi\rangle:=\sum_{0\leq k\leq d,M_k\neq\varnothing}\frac{1}{\mathcal{H}^k(M_k)}\int_{M_k}f_{M_k}^k\cdot\phi d\mathcal{H}^k,\quad\forall \phi\in C(V).
    \end{cases}
\end{equation}
We can derive the same result that the encoder \eqref{eq:full_encoder} has super-algebraic convergence. Such joint manifold function encoder is meaningful in practice. For example, in elasticity problems, displacement boundary conditions may be applied on manifolds of different dimensions, such as fixing a specific point and a specific curve, thus we can encode them simultaneously based on the encoder.

Moreover, the MFE exhibits consistency for manifold functions of different dimensions. Specifically, we provide a very simple illustrative example. Assume that $x\in \mathcal{M}_0$ is a point, $B(x,\frac{1}{i})\subset X_{d}$ is the $d$-dimensional closed ball of radius $\frac{1}{i}$ centered at $x$. One can readily derive that
\begin{equation}
    \lim_{i\to\infty}\frac{1}{\mathcal{H}^d(B(x,\frac{1}{i}))}\int_{B(x,\frac{1}{i})}1\cdot \phi_{n,m}d\mathcal{H}^d=\phi_{n,m}(x)=\frac{1}{\mathcal{H}^0(x)}\int_{x}1\cdot \phi_{n,m}d\mathcal{H}^0.
\end{equation}
The consistency of MFE is also important for joint manifold functions. Likewise, we use the elasticity problem as an example. In practical engineering problems, the effect of holding a single point of an object fixed should be nearly equivalent to that of holding fixed a small disk centered at that point. Consequently, their corresponding encoded vectors are expected to be approximately equal. 

\subsection{Manifold Functions with Different Measures}\label{sec:measure}

In practical engineering computations, functions on manifolds are often provided through discrete means, such as mesh-based and point cloud-based functions. To facilitate the integration of mesh/point cloud-based functions, we are inclined to perform integration with respect to the measure associated with the mesh/point cloud, rather than the Hausdorff measure. Consider the $L^2$-density measure space
\begin{equation}
\begin{split}
    \mathfrak{D}^2(M):=\bigg\{&\mu^f_M:\mu^f_M(A):=\int_{A}fd\mathcal{H}^k,A\subset M{\rm\ is\ }\mathcal{H}^k{\rm -measurable}, \\ &f\in L^2(M),f>0,\norm{f}_{L^1(M)}=1\bigg\},
\end{split}
\end{equation}
then a mesh or point cloud on the manifold $M$ corresponds to a random sample from some density measure in $\mathfrak{D}^2(M)$, so that the integral of a function on $M$ with respect to that measure is given by its Monte Carlo integration over the mesh or point cloud. Consequently, we can define the space of manifold functions with different measures, as
\begin{equation}
    X_k^\mu:=\{(M,\mu_M,f_{\mu_M}):M\in\mathcal{M}_k,\mu_M\in \mathfrak{D}^2(M), f_{\mu_M}\in L^2(M,\mu_M)\}.
\end{equation}
The encoding scheme for this space is not fundamentally different from that of the aforementioned case. The encoder is given by
\begin{equation}
    \begin{cases}
        \Phi_n^1:X_k^\mu\to\R^{\kappa(n)},\quad \Phi_n^1(M,\mu_M,f_{\mu_M}):=\left(\frac{1}{\mathcal{H}^k(M)}\int_{M}1\cdot \phi_{n,m}d\mathcal{H}^k\right)_{1\leq m\leq \kappa(n)}\in\R^{\kappa(n)}, \\ \Phi_n^2:X_k^\mu\to\R^{\kappa(n)},\quad\Phi_n^2(M,\mu_M,f_{\mu_M}):=\left(\int_{M}1\cdot\phi_{n,m}d\mu_M\right)_{1\leq m\leq \kappa(n)}\in\R^{\kappa(n)},\\  \Phi_n^3:X_k^\mu\to\R^{\kappa(n)},\quad\Phi_n^3(M,\mu_M,f_{\mu_M}):=\left(\int_{M}f_{\mu_M}\cdot\phi_{n,m}d\mu_M\right)_{1\leq m\leq \kappa(n)}\in\R^{\kappa(n)},\\
        \Phi_n:X_k^\mu\to\R^{3\kappa(n)},\quad \Phi_n(M,\mu_M,f_{\mu_M}):=(\Phi_n^i(M,\mu_M,f_{\mu_M}))_{i=1,2,3}\in\R^{3\kappa(n)}.
    \end{cases}
\end{equation}
In such case, assume that $\{x_i\}_{i=1}^N$ is a set of points sampled from $\mu_M$, then
\begin{equation}
    \int_{M}1\cdot\phi_{n,m}d\mu_M\approx\frac{1}{N}\sum_{i=1}^N\phi_{n,m}(x_i),\quad \int_{M}f_{\mu_M}\cdot\phi_{n,m}d\mu_M\approx\frac{1}{N}\sum_{i=1}^Nf_{\mu_M}(x_i)\cdot\phi_{n,m}(x_i),
\end{equation}
by Monte Carlo integration. In order to compute $\Phi_n^1$, we can apply the mesh information if it is available; otherwise, we need to resort to some estimation methods for uniform measure based on point cloud $\{x_i\}$, for examples \cite{coifman2006diffusion,fleishman2005robust}.

\subsection{Multiple Inputs and Outputs}

In above discussion about MFE for learning solution mappings of parametric PDEs, we show the results for single input and single output, in fact it readily generalizes to the multi-input and multi-output scenario. Let
\begin{equation}
\begin{split}
\mathcal{G}:X_{k_1}\times\cdots\times X_{k_n}&\to X_{\tilde{k}_1}\times\cdots\times X_{\tilde{k}_m} \\
(M_1,f_{M_1}^1)\times\cdots\times (M_n,f_{M_n}^n)&\mapsto(\tilde{M}_1,u_{\tilde{M}_1}^1)\times\cdots\times (\tilde{M}_m,u_{\tilde{M}_m}^m)
\end{split}
\end{equation}
be the multiple inputs and outputs solution mapping of PDE, where
\begin{equation}
    \tilde{M}_i=J_i(M_1,...,M_n),\quad 1\leq i\leq m,
\end{equation}
and $J_i$ are known mappings. Similar to the previous analysis, we can derive that such solution mapping is also learnable. Of course, we can replace the $k_i$-dimensional manifold function space $X_{k_i}$ by the joint manifold function space $X_i$, the same to the image spaces. In the following we will show a corresponding numerical example of Poisson equation in Section \ref{sec:poisson}. 

\subsection{Fourier Basis}

Another useful approximating sequence of bases is the Fourier basis. Within the theoretical framework established above, we need to prove its $(r,s,t)$-convergence property. However, the periodicity of the Fourier basis makes it impossible to converge in $C(V)$-norm for non-periodic function in $C^r(V)$. Fortunately, this issue is easily addressed simply by making the following adjustments:
\begin{itemize}
    \item Replace the involved spaces $C(V),C^r(V),H^s(V)$ by the periodic spaces $C(\mathbb{T}^d),C^r(\mathbb{T}^d),H^s(\mathbb{T}^d)$, where $\mathbb{T}^d:=\R^d/\mathbb{Z}^d$.
    \item Due to the constraint of periodicity, we have to replace $M_i\subset V$ by $M_i\subset\mathring{V}$ in the definition of $\mathcal{M}_k$, i.e. Eq. \eqref{eq:def_Mk}, to ensure that the manifolds do not touch the boundary of $V$.
\end{itemize}
With above two simple adjustments, all the analysis in previous sections holds. According to Theorem~1.1 of \cite{canuto1982approximation}, we know the sequence of Fourier basis is $(r,0,r-s)$-convergence for $\frac{d}{2}<s<r$ based on function spaces on $\mathbb{T}^d$, thus the related convergence rate for Fourier-based MFE is $\mathcal{O}(n^{s-r})$. 

An interesting case is the PCNO \cite{zeng2025point}. In their architecture, the first layer implicitly employs the encoding scheme of MFE based on Fourier basis. A slight difference is that, when encoding the manifold, the integration function they use is not the constant function ``$1$'', but rather the coordinate functions $\Pi_i(x):=x_i$ for $x=(x_1,...,x_d)\in M$, $1\leq i\leq d$. Basically, it is feasible according to Remark \ref{rem:feasible}. Nevertheless, from the perspective of encoding, the input coordinate information is redundant.

\section{Numerical Experiments}\label{sec:numerical_experiments}

In this section we show several numerical experiments to verify our theoretical results.

\subsection{Reconstruction of MFE}
Here, we begin by demonstrating the reconstruction performance of the MFE. Even though the decoded object is a dual in $(C^r(V))'$, we can still visualize it in the function space, by considering $\phi_{n,m}$ instead of $e_{n,m}^*$. As an example, we consider the boundary of a polygon, which is a 1-d Lipschitz manifold, together with a disk that is a 2-d Lipschitz manifold, contained in $[0,1]^2$. After generating random functions on them, we construct a joint manifold function as defined in Section \ref{sec:joint}. We employ the MFE based on Legendre polynomials, and present the reconstructions with different $n=8,16,32,64,128,256,512$ in Figure \ref{fig:reconstruction} (we employ a transformation $\log\left(\max(1,\phi)\right)$ for better visualization). With increasing $n$, the manifolds and functions appears increasingly clear. Overall, the values near polygonal boundaries are larger than those on the disk, since the values on the non-full-dimensional manifold will tend to infinity as $n\to\infty$.

\begin{figure}[htbp]
    \centering
    \includegraphics[width=0.99\textwidth]{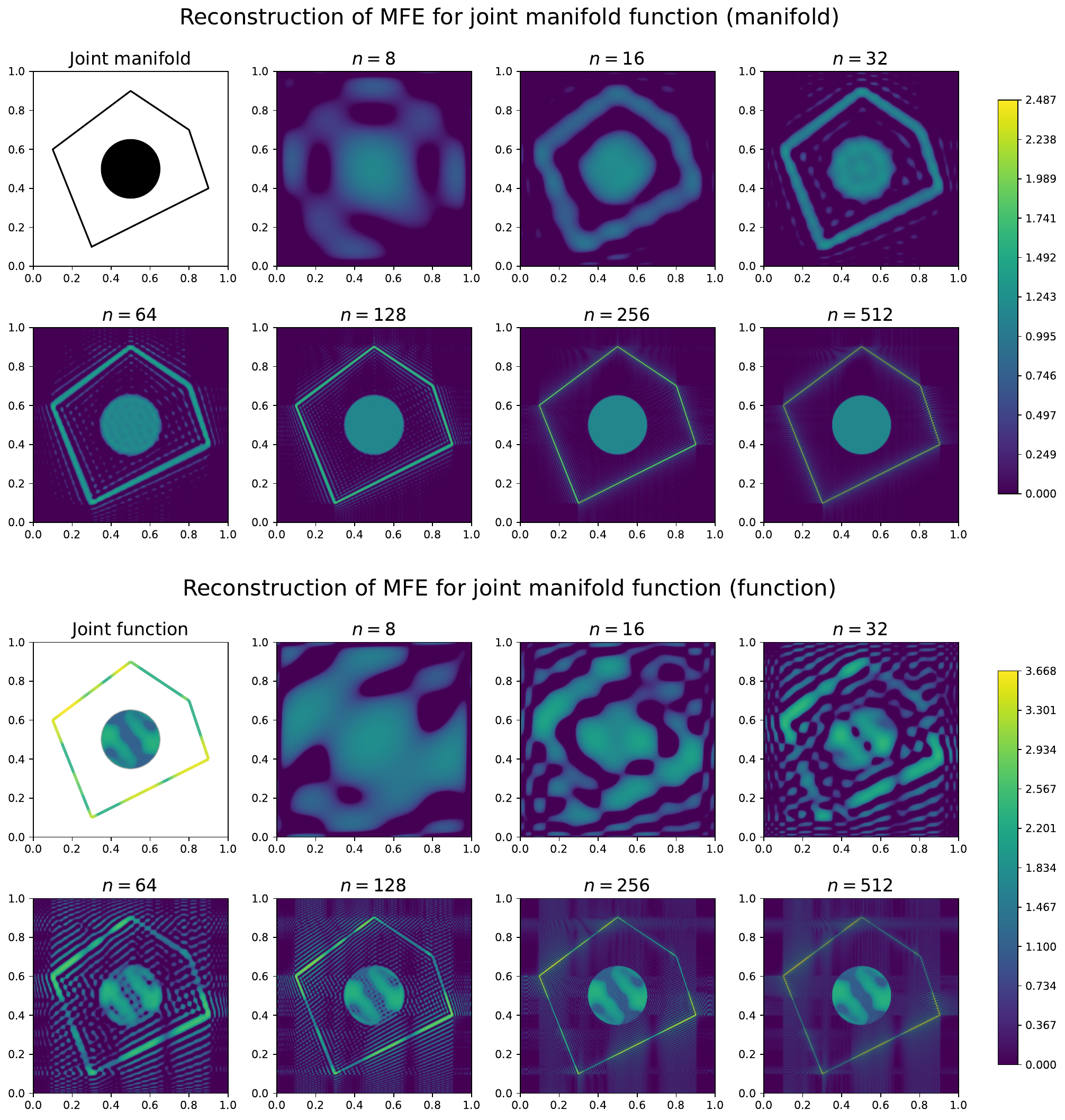}
    \caption{Reconstruction of a joint manifold function, consisting of a 1-d boundary of polygon and a 2-d disk. Two parts of the presented joint function are rescaled separately for better visualization.}
    \label{fig:reconstruction}
\end{figure}

\subsection{2-d Poisson Equation}\label{sec:poisson}
Consider the 2-d Poisson equation
\begin{equation}
    \begin{cases}
        -\nabla\cdot(k\nabla u)=f\quad &{\rm in\ }\Omega,\\
        u=g\quad &{\rm on\ }\partial\Omega.
    \end{cases}
\end{equation}
Assume that $\Omega$ is contained in $[0,1]^2$, we aim to learn the solution mapping
\begin{equation}
\begin{split}
    \mathcal{G}:X_2\times X_2\times X_1&\to X_2.\\ ((\Omega,k_\Omega),(\Omega,f_\Omega),(\partial\Omega,g_{\partial\Omega}))&\mapsto(\Omega,u_\Omega)
\end{split}
\end{equation}
Via the MFE based on Legendre polynomials, we obtain
\begin{equation}
    \begin{cases}
        \Phi_n^1(\Omega)=\left(\int_{\Omega}1\cdot (\ell_{i_1}\otimes\ell_{i_2})d\mathcal{H}^2\right)_{0\leq i_1,i_2\leq n-1}\in\R^{n^2}, \\
        \Phi_n^2(k_\Omega)=\left(\int_{\Omega}k_\Omega\cdot (\ell_{i_1}\otimes \ell_{i_2})d\mathcal{H}^2\right)_{0\leq i_1,i_2\leq n-1}\in\R^{n^2},\\
        \Phi_n^3(f_\Omega)=\left(\int_{\Omega}f_\Omega\cdot (\ell_{i_1}\otimes \ell_{i_2})d\mathcal{H}^2\right)_{0\leq i_1,i_2\leq n-1}\in\R^{n^2},\\
        \Phi_n^4(g_{\partial\Omega})=\left(\int_{\partial\Omega}g_{\partial\Omega}\cdot (\ell_{i_1}\otimes \ell_{i_2})d\mathcal{H}^1\right)_{0\leq i_1,i_2\leq n-1}\in\R^{n^2}.
    \end{cases}
\end{equation}
Since $k,f,g$ correspond to the same manifold $\Omega$, we can simply keep one encoded vector, as $\Phi_n^1(\Omega)$. To learn the solution mapping, we apply a basic operator learning model MIONet \cite{jin2022mionet}. The network can be written as
\begin{equation}
    \hat{\mathcal{G}}_\theta((\Omega,k_\Omega),(\Omega,f_\Omega),(\partial\Omega,g_{\partial\Omega})):= \mathcal{S}\left( \underbrace{\tilde{\mathbf{g}}_1(\Phi_{n}^1(\Omega))}_{\text{branch}_1} \odot\underbrace{\tilde{\mathbf{g}}_2(\Phi_{n}^2(k_\Omega))}_{\text{branch}_2}\odot \underbrace{\tilde{\mathbf{g}}_3(\Phi_{n}^3(f_\Omega),\Phi_{n}^4(g_{\partial\Omega}))}_{\text{branch}_3} \odot \underbrace{\tilde{\mathbf{u}}(\cdot)}_{\text{trunk}} \right),
\end{equation}
where $\odot$ is the Hadamard product (element-wise product), $\mathcal{S}$ is the summation of all the components of a vector,
\begin{equation}
    \tilde{\mathbf{g}}_1:\R^{n^2}\to\R^p,\quad \tilde{\mathbf{g}}_2:\R^{n^2}\to\R^p,\quad \tilde{\mathbf{u}}:\R^{2}\to\R^p
\end{equation}
are parameterized by nonlinear fully-connected neural networks, while $\tilde{\mathbf{g}}_3:\R^{2n^2}\to\R^p$ is parameterized by a linear mapping, i.e. a dense $p$-by-$2n^2$ matrix, due to the linearity of $\mathcal{G}$ with respect to $(f,g)$. Here $\theta$ includes all the parameters involved in $\tilde{\mathbf{g}}_1,\tilde{\mathbf{g}}_2,\tilde{\mathbf{g}}_3,\tilde{\mathbf{u}}$. 

We generate a dataset including a training set
\begin{equation}
    \T:=\left\{(\Omega_i,k_{\Omega_i}^i),(\Omega_i,f_{\Omega_i}^i),(\partial\Omega_i,g_{\partial\Omega_i}^i);(\Omega_i,u_{\Omega_i}^i)\right\}_{i=1}^N
\end{equation}
with $N=9000$ (another 1000 for test set), and the manifold functions are given based on meshes. Note that the domains in the dataset include disk-like domains and annulus-like domains, which are not homeomorphic to each other. Then we train the MIONet by loss function
\begin{equation}
    \mathcal{L}(\T;\theta):=\frac{1}{N}\sum_{i=1}^N\norm{\hat{\mathcal{G}}_\theta((\Omega_i,k_{\Omega_i}^i),(\Omega_i,f_{\Omega_i}^i),(\partial\Omega_i,g_{\partial\Omega_i}^i))-u_{\Omega_i}^i}_{L^2(\Omega_i)}^2.
\end{equation}
The sizes of the used fully-connected neural networks in MIONet are $[n^2,500,500,500,p]$ for $\tilde{\mathbf{g}}_1$, $[n^2,500,500,500,p]$ for $\tilde{\mathbf{g}}_2$, $[2n^2,p]$ for $\tilde{\mathbf{g}}_3$, and $[2,500,500,500,p]$ for $\tilde{\mathbf{u}}$, where $n=12$, $p=500$. We apply the Adam optimizer \cite{kingma2014adam} and train $5\times10^6$ iterations with batch size 5 and learning rate $10^{-5}$. The $L^2$ relative error on test set is $4.1\%$. We show some examples of prediction in Figure \ref{fig:poisson}.

\begin{figure}[htbp]
    \centering
    \includegraphics[width=0.90\textwidth]{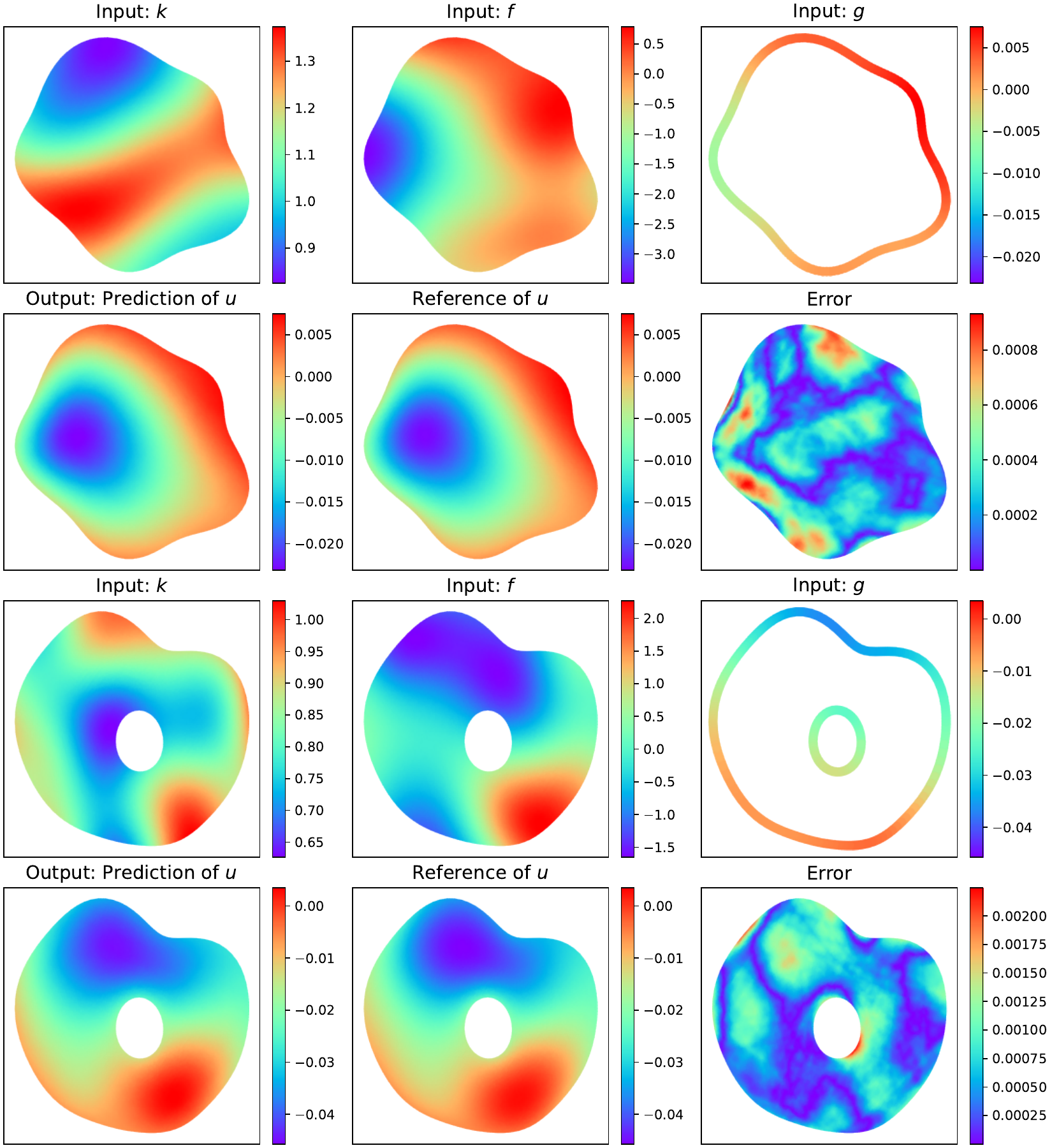}
    \caption{The prediction of the Poisson equation on disk-like domains and annulus-like domains.}
    \label{fig:poisson}
\end{figure}

\subsection{3-d Elasticity Problem}\label{sec:elasticity}

We take as our next example a more complicated scenario. Consider a standard rolling-element bearing in three-dimensional space. It consists of three annular components: an outer race, a middle race (or cage/separator), and an inner race, along with seven balls. The middle annular component houses the seven balls and retains them between the inner and outer races. Now, we secure the inner race of the bearing and apply a non-uniform random pressure over a randomly selected patch on the outer race. This will induce deformation in the bearing. Such deformation displacements can be obtained by solving the contact nonlinear elasticity equations. The solution mapping from the pressure to the displacement can be written as
\begin{equation}
\begin{split}
    \mathcal{G}:X_2&\to X_3\times X_3\times X_3, \\ (M,f_{M})&\mapsto((M_b,u_{M_b}^1),(M_b,u_{M_b}^2),(M_b,u_{M_b}^3))
\end{split}
\end{equation}
where $M$ is the 2-d manifold on which the pressure is applied, $M_b$ is the 3-d manifold of the whole bearing assembly, and $(u_{M_b}^1,u_{M_b}^2,u_{M_b}^3)$ is the solution (displacement) on the bearing ($u_{M_b}^i$ denotes the displacement component in the $i$-th dimension).

Again we employ the MFE based on Legendre polynomials, as
\begin{equation}
    \begin{cases}
        \Phi_n^1(M)=\left(\int_{M}1\cdot (\ell_{i_1}\otimes\ell_{i_2}\otimes \ell_{i_3})d\mathcal{H}^2\right)_{0\leq i_1,i_2,i_3\leq n-1}\in\R^{n^3}, \\
        \Phi_n^2(f_M)=\left(\int_{M}f_M\cdot (\ell_{i_1}\otimes\ell_{i_2}\otimes \ell_{i_3})d\mathcal{H}^2\right)_{0\leq i_1,i_2,i_3\leq n-1}\in\R^{n^3}.\\
    \end{cases}
\end{equation}
To learn this solution mapping, here we use the multiple-output version of MIONet \cite{jin2022mionet}, written as
\begin{equation}
    \hat{\mathcal{G}}_\theta(M,f_M):= W\cdot\left( \underbrace{\tilde{\mathbf{g}}_1(\Phi_{n}^1(M))}_{\text{branch}_1} \odot\underbrace{\tilde{\mathbf{g}}_2(\Phi_{n}^2(f_M))}_{\text{branch}_2}\odot \underbrace{\tilde{\mathbf{u}}(\cdot)}_{\text{trunk}} \right),
\end{equation}
where $\tilde{\mathbf{g}}_1:\R^{n^3}\to\R^p$, $\tilde{\mathbf{g}}_2:\R^{n^3}\to\R^p$, $\tilde{\mathbf{u}}:\R^{3}\to\R^p$ are modeled by fully-connected neural networks, and $W\in\R^{3\times p}$ is a trainable matrix. Note that here $\theta$ includes all the parameters involved in $\tilde{\mathbf{g}}_1,\tilde{\mathbf{g}}_2,\tilde{\mathbf{u}}$ and $W$.

We generate a dataset including a training set
\begin{equation}
    \T:=\left\{(M_i,f_{M_i}^i);(M_b,u_{M_b}^{1,i}),(M_b,u_{M_b}^{2,i}),(M_b,u_{M_b}^{3,i})\right\}_{i=1}^N
\end{equation}
with $N=800$ (another 200 for test set), and the manifold functions are given based on meshes. Then we train the MIONet by loss function
\begin{equation}
    \mathcal{L}(\T;\theta):=\frac{1}{N}\sum_{i=1}^N\norm{\hat{\mathcal{G}}_\theta(M_i,f_{M_i}^i)-(u_{M_b}^{1,i},u_{M_b}^{2,i},u_{M_b}^{3,i})}_{L^2(M_b,\R^3)}^2.
\end{equation}
The sizes of the used fully-connected neural networks in MIONet are $[n^3, 500, 500, 500, p]$ for $\tilde{\mathbf{g}}_1$, $[n^3, 500, 500, 500, p]$ for $\tilde{\mathbf{g}}_2$, and $[3, 500, 500, 500, p]$ for $\tilde{\mathbf{u}}$, where $n=12$, $p=500$ (the dimension of the encoded vector is only $12^3\times 2=3456$). We apply the Adam optimizer \cite{kingma2014adam} and train $2\times10^6$ iterations with batch size 10 and learning rate $10^{-5}$. The $L^2$ relative error on test set is $8.6\%$. It is noticed that the CAE software (Abaqus) require $10\sim20$ minutes to solve this problem, while the neural network model delivers predictions in under a second. We show an example of prediction in Figure \ref{fig:bearing}. Note that the ``Disp. mag.'' in this figure denotes the displacement magnitude $\sqrt{(u_{M_b}^1)^2+(u_{M_b}^2)^2+(u_{M_b}^3)^2}$.

\begin{figure}[htbp]
    \centering
    \includegraphics[width=1.0\textwidth]{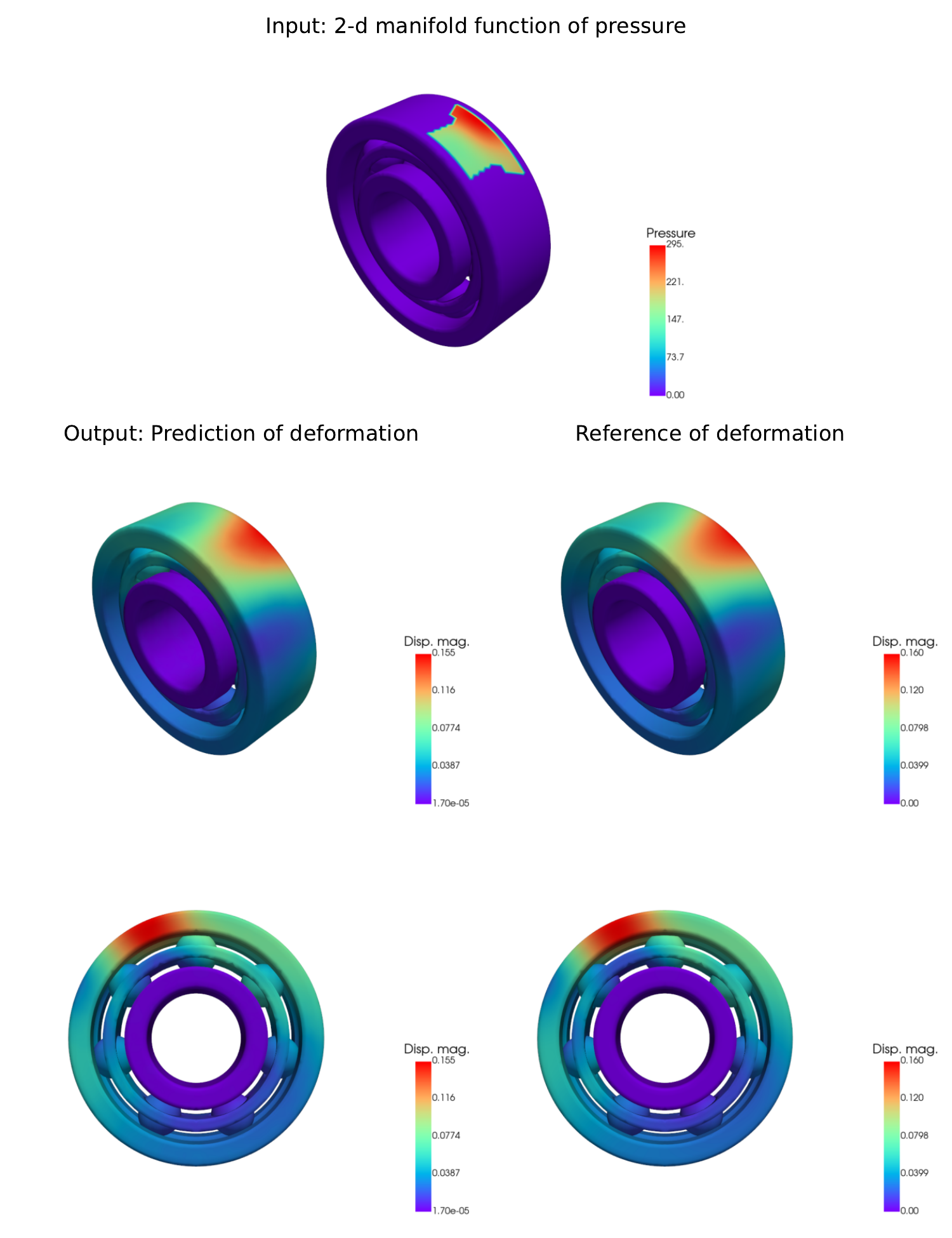}
    \caption{The prediction of the contact nonlinear elasticity problem on the bearing.}
    \label{fig:bearing}
\end{figure}

\section{Conclusions}\label{sec:conclusions}

In this work, we propose the Manifold Function Encoder (MFE) for identifying different functions defined on different manifolds. From the perspective of dual space, we prove that MFE achieves super-algebraic convergence based on smooth bases commonly used in spectral methods, such as Legendre polynomials and Fourier basis. We also show the approximation theory for MFE-based operator learning, in particular learning the solution mappings of PDEs defined on varying domains. Furthermore, we extend MFE to handle more complex cases, including encoding the union of manifold functions of different dimensions (joint manifold functions), encoding manifold functions with different measures (related to unstructured mesh and non-uniform point cloud in practice), and dealing with multiple inputs and outputs operator learning problems. Several numerical experiments demonstrate the high efficiency of MFE, especially the 3-d elasticity problem on the real-world bearing, where we encode the manifold function with only $12^3\times 2=3456$ parameters.

There remain some outstanding issues that are yet to be resolved. To complete the learning theory of parametric PDEs, we need to transfer the approximation result on ``manifold function space to Banach space'' to ``manifold function space to manifold function space'' and verify the continuity of the specific PDE problems under corresponding setting, as stated in Section \ref{sec:MFE_operator_learning}. We leave them for future work.



\bibliographystyle{abbrv}
\bibliography{references}

\end{document}